\documentclass{amsart}
\usepackage{amssymb, amsmath, latexsym, mathabx, dsfont,tikz}
\usepackage{multirow}

\renewcommand{\baselinestretch}{\baselinestretch}
\renewcommand{\baselinestretch}{1.1}
\numberwithin{equation}{section}

\newtheorem{thm}{Theorem}[section]
\newtheorem{lem}[thm]{Lemma}
\newtheorem{cor}[thm]{Corollary}
\newtheorem{prop}[thm]{Proposition}

\theoremstyle{definition}
\newtheorem{defn}[thm]{Definition}

\theoremstyle{remark}

\numberwithin{equation}{section}

\newcommand{\gen}{\text{gen}}

\newcommand{\ord}{\text{ord}}

\newcommand{\z}{{\mathbb Z}}
\newcommand{\q}{{\mathbb Q}}
\newcommand{\f}{{\mathbb F}}

\newcommand{\n}{{\mathbb N}}
\newcommand{\e}{{\epsilon}}

\begin{document}
\title[The number of representations of squares]{The number of representations of squares by integral quaternary quadratic forms}

\author{Kyoungmin Kim}

\address{Department of Mathematics, Sungkyunkwan University, Suwon 16419, korea}
\email{kiny30@skku.edu}
\thanks{This work was supported by the National Research Foundation of Korea(NRF) grant funded by the Korea government(MSIT) (NRF-2016R1A5A1008055  and NRF-2018R1C1B6007778)}



\begin{abstract}
Let $f$ be a positive definite (non-classic) integral quaternary quadratic form. We say $f$ is {\it strongly $s$-regular} if it satisfies a regularity property on the number of representations of squares of integers. In this article, we prove that there are only finitely many strongly $s$-regular quaternary quadratic forms up to isometry if the minimum of the nonzero squares that are represented by the quadratic form is fixed. Furthermore, we show that there are exactly $34$  strongly $s$-regular diagonal quaternary quadratic forms representing one (see Table $1$). In particular, we use eta-quotients to prove the strongly $s$-regularity of the quaternary quadratic form $x^2+2y^2+3z^2+10w^2$, which is, in fact, of class number $2$ (see Lemma \ref{120} and Proposition \ref{333}).
\end{abstract} 

\maketitle
\section{Introduction}
For a positive definite (non-classic) integral quadratic form of rank $k$
$$ 
f(x_1,x_2,\dots,x_k)=\sum_{i,j=1}^{k}a_{ij}x_ix_j, \quad (a_{ij}=a_{ji})
$$
we define the discriminant $df$ of $f$ to be the determinant of the symmetric matrix
$$
M_f=\left(\frac{\partial^2f}{\partial x_i\partial x_j}\right).
$$
For a positive integer $n$, we define $r(n,f)$ the number of representations of $n$ by $f$, that is, 
$$
r(n,f)=\vert\{(x_1,x_2,\dots,x_k)\in \z^k\mid f(x_1,x_2,\dots,x_4)=n\}\vert.
$$
It is well known that $r(n,f)$ is finite if $f$ is positive definite.

Let $f$ be a positive definite (non-classic) integral quadratic form of rank $k$. Let $n_1$ and $n_2$ be positive integers such that $P(n_1)\subset P(2df)$, $(n_2,2df)=1$ and $n=n_1n_2$. Here  $P(n)$ denotes the set of prime factors of $n$. 
The quadratic form $f$ is called {\it strongly $s$-regular} if for any positive integer $n=n_1n_2$,
$$
r(n_1^2n_2^2,f)=r(n_1^2,f)\cdot \prod_{p\nmid 2df}h_p(df,\mu_p),
$$
where $\mu_p=\text{ord}_p(n)$ for any prime $p$ and
$$
h_p(df,\mu_p)\!=\!
\begin{cases}
\displaystyle \sum_{t=0}^{2\mu_p}\left(\frac{(-1)^{\frac{k}2}df}p\right)^tp^{\frac{(k-2)t}{2}}
\ \text{if $k$ is even},\\[0.5cm]

\displaystyle \!\left(\frac{p^{(k-2)(\mu_p +1)}-1}{p^{k-2}-1}-p^{\frac{k-3}{2}} \left(\frac{(-1)^{\frac{k-1}{2}}df}{p}\right)\frac{p^{(k-2)\mu_p}-1}{p^{k-2}-1}\right)\ \!\text{otherwise}.
\end{cases}
$$
We say the genus of a quadratic form $f$ is {\it indistinguishable by squares} if for any integer $n$, $r(n^2,f)=r(n^2,f')$ for any $f'$ in the genus of $f$. From the definition, if the class number of a quadratic form $f$ is one, then the genus of $f$ is indistinguishable by squares. It is unknown whether or not the number of strongly $s$-regular quadratic forms with given rank is finite. Also, it might be interesting to find all such quadratic forms. Related with these questions, there are the following results.

We proved in \cite{ko1} that every ternary quadratic form in the genus of a ternary quadratic form $f$ is strongly $s$-regular if and only if the genus of a ternary quadratic form $f$ is indistinguishable by squares. Furthermore, we completely resolved the conjecture given by Cooper and Lam in \cite{cl}.

It was proved in \cite{ko2} that every strongly $s$-regular ternary quadratic form represents all squares that are represented by its genus, and there are only finitely many strongly $s$-regular ternary quadratic forms up to isometry if 
$$
m_s(f)=\text{min}_{n\in\z^{+}}\{n:r(n^2,f)\ne 0 \}
$$ 
is fixed. Furthermore, it was proved that there are exactly $207$ strongly $s$-regular ternary quadratic forms that represent one.

In this article, we consider the strongly $s$-regular quaternary quadratic forms.
we prove that if the genus of a quaternary quadratic form $f$ is indistinguishable by squares, then every quaternary quadratic form in the genus of $f$ is strongly $s$-regular. We also prove that any strongly $s$-regular quaternary quadratic form represents all squares of integers that are represented by its genus, and there are only finitely many strongly $s$-regular quaternary quadratic forms up to isometry if 
$$
m_s(f)=\text{min}_{n\in\z^{+}}\{n:r(n^2,f)\ne 0 \}
$$ 
is fixed. Furthermore, we show that there are exactly $34$ strongly $s$-regular diagonal quaternary quadratic forms representing one up to isometry (see Table $1$). In particular, we use eta-quotients to prove the strongly $s$-regularity of the quadratic form $x^2+2y^2+3z^2+10w^2$ (see Lemma \ref{120} and Proposition \ref{333}). We also use the mathematics software MAPLE to prove Lemma \ref{357} and Theorem \ref{main2}.

The term lattice will always refer to a positive definite non-classic integral $\z$-lattice on an $n$-dimensional positive definite quadratic space over $\q$. Here, a $\z$-lattice is called non-classic integral if the norm ideal $\mathfrak nL$ is $\z$.  Let $L=\z x_1+\z x_2+ \dots+\z x_n$ be a $\z$-lattice of rank $n$. We write
$$
L\simeq (B(x_i,x_j)).
$$
Here, $B$ is the associated bilinear form. The right hand side matrix is called a {\it matrix presentation} of
$L$. If $B(x_i,x_j)=0$ for any $i\ne j$, then we write $L \simeq \langle Q(x_1),Q(x_2),\dots,Q(x_n)\rangle$, where $Q$ is the quadratic map such that $Q(x)=B(x,x)$ for any $x \in L$. 
We define 
$$
w(L)=\sum_{[M] \in \gen(L)} \frac1{o(M)}\quad \text{and} \quad   r(n,\gen(L))= \frac 1 {w(L)}  \sum_{[M] \in \gen(L)} \displaystyle \frac{r(n,M)}{o(M)},
$$
where $[M]$ is the equivalence class containing $L$ in the genus of $L$ and $o(L)$ is the order of the isometry group $O(L)$.
We say an integer $n$ is represented by the genus of $L$ if an integer $n$ is represented by $L$ over $\z_p$ for any prime $p$ including the infinite prime.
We always assume that $\Delta_p$ is a nonsquare unit in $\z_p^{\times}$ for any odd prime $p$.

Any unexplained notations and terminologies can be found in  \cite {ki} or \cite {om}.

\section{Representations of squares by quaternary quadratic forms}
In this section, we investigate the relation between the strongly $s$-regularity of a quaternary quadratic form and the indistinguishable genus of a quaternary quadratic form by squares. 

\begin{defn}\label{defa}
Let $L$ be a quaternary $\z$-lattice. Let $n_1$ and $n_2$ be positive integers such that $P(n_1)\subset P(2dL)$, $(n_2,2dL)=1$ and $n=n_1n_2$. Here  $P(n)$ denotes the set of prime factors of $n$. 
The quaternary $\z$-lattice $L$ is called {\it strongly $s$-regular} if for any positive integer $n=n_1n_2$,
$$
r(n_1^2n_2^2,L)=r(n_1^2,L)\cdot \prod_{p\nmid 2dL}h_p(dL,\mu_p),
$$
where $\mu_p=\text{ord}_p(n)$ for any prime $p$ and
$$
h_p(dL,\mu_p)=\sum_{t=0}^{2\mu_p}\left(\frac{dL}p\right)^tp^{t}.
$$
\end{defn}

\begin{defn}
Let $L$ be a quaternary $\z$-lattice.
We say the genus of $L$ is {\it indistinguishable by squares} if for any integer $n$, $r(n^2,L)=r(n^2,L')$ for any $\z$-lattice $L'$ in the genus of $L$.
\end{defn}

\begin{lem}
Let $L$ be a quaternary $\z$-lattice and $V=\mathbb Q \otimes L$ be the quadratic space.
Then $L$ represents at least one square of an integer.
\end{lem}
\begin{proof}
The lemma follows directly from the fact that $r(n^2,L)\ne 0$ for some integer $n$  if and only if $1$ is represented by the quadratic space $V$.
\end{proof}

\begin{lem}\label{Min}
Let $L$ be a quaternary $\z$-lattice and let $n$ be a positive integer. For any prime $p$, we let $\ord_p(n)=\mu_p$. If $1$ is represented by the genus of $L$, then we have 
$$
\begin{array}  {rl}
\displaystyle \frac{r(n,\text{gen}(L))}{r(1,\gen(L))}
&=n\displaystyle  \prod_{p\vert 2dL} \frac{\alpha_{p}(n,L)}{\alpha_{p}(1,L)}\displaystyle  \prod_{p\nmid 2dL} \frac{\alpha_{p}(n,L)}{\alpha_{p}(1,L)}\\
&=\displaystyle \prod_{p\vert 2dL} p^{\mu_p} \cdot \frac{\alpha_{p}(n,L)}{\alpha_{p}(1,L)}
\displaystyle \prod_{p\nmid 2dL}\left(\sum_{t=0}^{\mu_p}\left(\frac{dL}{p}\right)^{\mu_p-t}p^t \right).\\
\end{array}
$$
In particular, if the $\z$-lattice $L$ has class number $1$, then we have
$$
r(n,L)=r(1,L)\prod_{p\vert 2dL}p^{\mu_p}\cdot \frac{\alpha_{p}(n,L)}{\alpha_{p}(1,L)}
\prod_{p\nmid 2dL} \left( \sum_{t=0}^{\mu_p}\left(\frac{dL}{p}\right)^{\mu_p-t}p^t\right).
$$
Here, $\alpha_p$ is the local density which is defined in Chapter 5 of \cite{ki}.
\end{lem}

\begin{proof}
By the Minkowski--Siegel formula in Chapter 6 of \cite{ki}, we have
$$
r(n,\text{gen}(L))=\pi^{2}\cdot\frac{4}{\sqrt{dL}}\cdot n \cdot \prod_{p<\infty}\alpha_p({n,L}),
$$
where $\alpha_p$ is the local density. If $p$ does not divide $2dL$, then we have
$$
\alpha_p(n,L)=\displaystyle\frac{\left(p-\left(\frac{dL}{p}\right)^{\mu_p+1}\frac1{p^{\mu_p}}\right)\cdot \left(1-\left(\frac{dL}{p}\right)\frac1{p^2}\right)}{p-\left(\frac{dL}{p}\right)},
$$
by Theorem 3.1 of \cite{ya}. Therefore, for any prime $p$ not dividing $2dL$, we have
$$
p^{\mu_p}\cdot\frac{\alpha_{p}(n,L)}{\alpha_{p}(1,L)}=\displaystyle\frac{\left(p^{\mu_p+1}-\left(\frac{dL}{p}\right)^{\mu_p+1}\right)}{p-\left(\frac{dL}{p}\right)}=\sum_{t=0}^{\mu_p}\left(\frac{dL}{p}\right)^{\mu_p-t}p^t.
$$
The lemma follows from this.
\end{proof}

Let $L$ be a quaternary $\z$-lattice and let $p$ be a prime. Suppose that $p$ does not divide $dL$. Then $L_p$ is $\z_p$-maximal lattice. By 82:23 of \cite{om}, we have $L_p=H\perp G$, where  $H$ is an orthogonal sum of hyperbolic and $G$ is either $0$ or anisotropic. Let $\{e_1,\dots, e_{2r}\}$ be a basis of $H$ satisfying $Q(e_i)=0$ for any $i=1,\dots, 2r$,  $B(e_{2j-1},e_{2j})=\frac12$ for any $j=1,\dots, r$ and $B(e_k,e_l)=0$ otherwise. If $G\neq 0$, then let $\{e_{2r+1},e_{2r+2}\}$ be a basis of $G$. Such a basis $\{e_1,\dots,e_{2r},e_{2r+1},e_{2r+2}\}$ of $L_p$ will be called a standard basis.
We define $R_p(L)$ to be the set of all lattice $K$ in the genus of $L$ such that $K_q=L_q$ for any prime $q\neq p$ and there is a standard basis $\{e_1,\dots,e_{2r},e_{2r+1},e_{2r+2}\}$ of $L$ satisfying
$$
\{p^{-1}e_1,pe_2,\dots,p^{-1}e_{2r-1},pe_{2r},e_{2r+1},e_{2r+2}\} 
$$
is a basis of $K_p$. We put $c_p(L)=\vert R_p(L) \vert$. Furthermore, for a primitive vector $x\in L_p$ with $Q(x)\in p^2\z_p$, we define
$$
\rho_p(L)=\vert\{K \in R_p(L) : p^{-1}x\in K_p\} \vert.
$$
Note that $\rho_p(L)$ is independent of the choice of $x$.

\begin{lem}\label{indi}
Let $L$ be a quaternary $\z$-lattice. For every $L'$ in the genus of $L$, if $r(n^2,L)=r(n^2,L')$ for every integer $n$ such that every prime factor of $n$ divides $dL$, then the genus of $L$ is indistinguishable by squares.
\end{lem}
\begin{proof}
Let $L$ be a quaternary $\z$-lattice. Then by Hilfssatz $1$ of \cite{sp}, the action of Hecke operators $T(p^2)$ for any prime $p\nmid dL$ on theta series of the lattice $L$ gives
\begin{equation}\label{hecke1}
\begin{array}{ll}
&\hspace{-1cm}\displaystyle r(p^2n,L)+p\cdot\left(\frac{dL}{p}\right)\cdot r(n,L)+p^2\cdot r\left(\frac{n}{p^2},L\right)\\[0.35cm]
&\hspace{2cm} \displaystyle =\frac{1}{\rho_p(L)}\sum_{K\in R_p(L)}r(n,K)+\left(\kappa_p(L)-\frac{c_p(L)}{\rho_p(L)}\right)r(n,L).
\end{array}
\end{equation}
Here, $\kappa_p(L)=p^2+p\cdot\left(\frac{dL}{p}\right)+1$.
Suppose that $r(n^2,L)=r(n^2,L')$ for every $L'\in \gen(L)$. Then, by \eqref{hecke1}, for any prime $p\nmid dL$, we have
$$
r(p^2n^2,K)=(p^2+1)r(n^2,K)-p^2\cdot r\left(\frac{n^2}{p^2},K\right),
$$
for every $K\in \gen(L)$. Hence, we have $r(p^2n^2,L)=r(p^2n^2,L')$ for every $L'\in \gen(L)$ and for any prime $p\nmid dL$. The lemma follows from induction on the number of prime factors not dividing $dL$ counting multiplicity.
\end{proof}

\begin{prop}\label{indis}
Let $L$ be a quaternary $\z$-lattice. If the genus of $L$ is indistinguishable by squares, then every $\z$-lattice in the genus of $L$ is strongly $s$-regular.
\end{prop}
\begin{proof}
Suppose that $r(n^2,L)=r(n^2,L')$ for every integer $n$ and every $L'$ in the genus of $L$. 
Let $n_1$ and $n_2$ be positive integers such that $P(n_1)\subset P(2dL)$, $(n_2,2dL)=1$ and $n=n_1n_2$. First, assume that $r(n_1^2,L)\ne 0$. By the Minkowski--Siegel formula (see, for example, Lemma \ref{Min}), we have
$$
\begin{array}{ll}
\displaystyle\frac{r(n_1^2n_2^2,L')}{r(n_1^2,L')}=\frac{r(n_1^2n_2^2,\gen(L))}{r(n_1^2,\gen(L))}&=\displaystyle \prod_{p\vert 2dL} \frac{\alpha_{p}(n_1^2n_2^2,L)}{\alpha_{p}(n_1^2,L)}\prod_{p\nmid 2dL}h_p(dL,\mu_p)\\
&=\displaystyle\prod_{p\nmid 2dL}h_p(dL,\mu_p),
\end{array}
$$
for every $\z$-lattice $L'\in \gen(L)$. Next, assume that $r(n_1^2,L)=0$. Then by Lemma \ref{indi}, we see that $r(n_1^2n_2^2,L')=0$ for every $\z$-lattice $L'\in \gen(L)$. Therefore, we have
$$
r(n_1^2n_2^2,L')=r(n_1^2,L')\displaystyle\prod_{p\nmid 2dL'}h_p(dL',\mu_p),
$$
which implies that every $\z$-lattice $L'$ in the genus of $L$ is strongly $s$-regular.  
\end{proof}

Now, we present some known results on the number of representations of integers by quadratic forms, which are needed later.
Let $L$ be a (non-classic integral) quaternary $\z$-lattice. For any prime $p$, the $\lambda_p$-transformation (or Watson transformation) is defined as follows: 
$$
\Lambda_p(L)= \{ x \in L : Q(x + z) \equiv Q(z) \  (\text{mod} \ p) \mbox{ for
all $z \in L$}\}.
$$
Let $\lambda_p(L)$ be the non-classic integral
lattice obtained from $\Lambda_p(L)$ by scaling $V=L\otimes \mathbb Q$ by a suitable
rational number. For a positive integer $N=p_1^{e_1}p_2^{e_2}\cdots p_k^{e_k}$,   we  also define 
$$
\lambda_N(L)=\lambda_{p_1}^{e_1}(\lambda_{p_2}^{e_2}(\cdots\lambda_{p_{k-1}}^{e_{k-1}}(\lambda_{p_k}^{e_k}(L))\cdots)).
$$
Note that $\lambda_p(\lambda_q(L))=\lambda_q(\lambda_p(L))$ for any primes $p \ne q$.

\begin{lem}\label{aniso}
Let $L$ be a quaternary $\z$-lattice and let $p$ be an odd prime. If the unimodular component in a Jordan decomposition of $L_p$ is anisotropic, then
$$
r(pn,L)=r(pn,\Lambda_p(L)).
$$
\begin{proof}
See \cite{co}.
\end{proof}
\end{lem}

Let $L$ be a ternary $\z$-lattice. Assume that the $\frac12\z_p$-modular component in a Jordan decomposition of $L_p$ is nonzero isotropic. Assume that $p$ is a prime dividing $\frac12dL$.
Then by Weak Approximation Theorem, there exists a basis $\{x_1, x_2, x_3 \}$ for $L$ such that
$$
(B(x_i,x_j))\equiv\begin{pmatrix}0&\frac12\\ \frac12&0\end{pmatrix}\perp \langle p^{\ord_p(\frac12dL)} \delta\rangle \ (\text{mod} \ p^{\ord_p(\frac12dL)+1}),
$$
where $\delta$ is an integer not divisible by $p$. We define
$$
\Gamma_{p,1}(L) = \z px_1 + \z x_2+ \z x_3 \quad \text{and} \quad \Gamma_{p,2}(L) = \z x_1 + \z px_2+ \z x_3.
$$  
Note that $\Gamma_{p,1}(L)$ and $\Gamma_{p,2}(L)$ are unique sublattices of $L$ with index $p$ whose norm is contained in $p\z$. For some properties of these sublattices of $L$, see \cite {jlo}. 

\begin{lem} \label{iso}
Under the same assumptions given above,  we have 
$$
r(pn,L)=r(pn,\Gamma_{p,1}(L))+r(pn,\Gamma_{p,2}(L))-r(pn,\Lambda_p(L)).
$$ 
\end{lem}

\begin{proof}
See Proposition 4.1 of \cite{jlo}.
\end{proof}

\section{Strongly $s$-regular quaternary lattices}

\begin{lem}\label{squareregular}
Any strongly $s$-regular quaternary $\z$-lattice $L$ represents all squares of integers that are represented by its genus.
\end{lem}
\begin{proof}
Let $L$ be a strongly $s$-regular quaternary $\z$-lattice. Suppose, on the contrary, that there is an integer $a$ such that $a^2$ is represented by the genus of $L$, whereas it is not represented by $L$ itself.
Then for any prime $p\nmid 2dL$, if $a=p^t\cdot b$ for some integer $b$ such that $(b,p)=1$, then, for an integer $s\ge1$, we have
$$
r(p^{2s}a^2,L)=r(b^2,L)\left(\sum_{i=0}^{2t+2s}\left(\frac{dL}{p}\right)^ip^i\right)=0.
$$
By applying this to any prime $q$ such that $(q,2dL)=1$, we have
$$
r(n^2a^2,L)=0,
$$
for any integer $n$ such that $(n,2dL)=1$.
However, by Theorem 6.3 of \cite{ha}, there is a sufficiently large integer $m$ such that $(m,2dL)=1$ and
$$
r(m^2a^2,L)\ne 0,
$$ 
which is a contradiction.
\end{proof}

\begin{cor}
Let $L$ be a strongly $s$-regular quaternary $\z$-lattice. Then every integer $m$ such that $m^2$ is represented by $L$ is a multiple of 
$$
m_s(L)=\text{min}_{n\in\z^{+}}\{n:r(n^2,L)\ne 0 \}.
$$
\end{cor}
\begin{proof}
The corollary follows directly  from the fact that for any prime $p$, $\ord_p(m_s(L))$ is completely determined by $L_p$ by Lemma \ref{squareregular}.
\end{proof}

\begin{prop}\label{odd}
Let $q$ be an odd prime and let $L$ be a quaternary $\z$-lattice such that $L_q$ does not represent $1$. Assume that $L_q\simeq \langle \Delta_q, q^{\alpha}\e_1, q^{\beta}\e_2, q^{\gamma}\e_3 \rangle$ for $\e_1,\e_2,\e_3\in \z^{\times}_q$ and $1\le \alpha \le \beta \le \gamma$.

\begin{itemize}

\item [(i)] If $\alpha \ge 2$ and $L_q\nsimeq \langle \Delta_q, q^{2}\e_1, q^{2}\e_2, q^{2}\e_3 \rangle$ for some $\e_1,\e_2,\e_3\in \z^{\times}_q$, then $L$ is strongly $s$-regular if and only if $\lambda_q(L)$ is strongly $s$-regular. Furthermore, if one of them is true, then $m_s(L) = q\cdot m_s(\lambda_q(L))$.

\item [(ii)] If $\alpha = 1$ and $L_q\nsimeq \langle \Delta_q, q, -q, q^{\gamma}\e_3 \rangle$  for any $\gamma \ge 2$ and for some $\e_3\in \z^{\times}_q$ or $L_q\nsimeq \langle \Delta_q, q\e_1, q\e_2, q\e_3 \rangle$ for some $\e_1,\e_2,\e_3\in \z^{\times}_q$,
then $L$ is strongly $s$-regular if and only if $\lambda^2_q(L)$ is strongly $s$-regular. Furthermore, if one of them is true, then $m_s(L) = q\cdot m_s(\lambda^2_q(L))$.

\end{itemize}
\end{prop}
\begin{proof}
Since the proof is quite similar to each other, we only provide the proof of the first case.
For any positive integer $n$, let $n_1$ and $n_2$ be positive integers such that $P(n_1)\subset P(2dL)$, $(n_2,2dL)=1$ and $n=n_1n_2$, where $P(n)$ denotes the set of prime factors of $n$. Suppose that $L$ is strongly $s$-regular. Then we have
$$
r(q^2n_1^2n_2^2,L)=r(q^2n_1^2,L)\cdot\prod_{p\nmid 2dL}h_p(dL,\mu_p),
$$
where $\mu_p$ and $h_p(dL,\mu_p)$ are defined in Definition \ref{defa}. By Lemma \ref {aniso}, we have  
$$
r(q^2n_1^2n_2^2,L)=r(n_1^2n_2^2,\lambda_q(L)) \quad \text{and}\quad r(q^2n_1^2,L)=r(n_1^2,\lambda_q(L)).
$$ 
Hence we have
$$
r(n_1^2n_2^2,\lambda_q(L))=r(n_1^2,\lambda_q(L))\prod_{p\nmid 2dL}h_p(dL,\mu_p).
$$
Since $\alpha \ge 2$ and $L_q\nsimeq \langle \Delta_q, q^{2}\e_1, q^{2}\e_2, q^{2}\e_3 \rangle$ for some $\e_1,\e_2,\e_3\in \z^{\times}_q$ from the assumption, we see that
the set of primes dividing $2dL$ equals to the set of primes dividing $2d(\lambda_q(L))$. Hence, the above equation implies that $\lambda_q(L)$ is strongly $s$-regular.

Conversely, Suppose that $\lambda_q(L)$ is strongly $s$-regular.  Then we have
$$
r(n_1^2n_2^2,\lambda_q(L))=r(n_1^2,\lambda_q(L))\prod_{p\nmid 2d\lambda_q(L)}h_p(d\lambda_q(L),\mu_p).
$$
Hence if $\ord_q(n_1)\ge1$, then 
$$
r(n_1^2n_2^2,L)=r(n_1^2,L)\prod_{p\nmid 2dL}h_p(dL,\mu_p).
$$
Note that if $\ord_q(n_1)=0$, then $r(n_1^2n_2^2,L)=r(n_1^2,L)=0$ from the assumption.
Therefore $L$ is a strongly $s$-regular.  

Now assume that $L$ or $\lambda_q(L)$ is strongly $s$-regular. From the assumption, we know that
$m_s(L)$ is divisible by $q$. By Lemma \ref{aniso}, we also have  $r(q^2n,L)=r(n,\lambda_q(L))$. Therefore, we have $m_s(L)=q\cdot m_s(\lambda_q(L))$.
\end{proof}

\begin{prop}\label{even}
Let $L$ be a quaternary $\z$-lattice such that $L_2$ does not represent $1$. Assume that $L_2\simeq \langle \e_1 , 2^{\alpha}\e_2 \rangle \perp M$ for $\e_1,\e_2\in \z_2^{\times}$ and $\alpha \ge 0$.
\begin{itemize}
\item [(i)] If $\alpha \ge 2$ and $M$ is an improper modular lattice with norm contained in $4\z_2$ or $M \simeq \langle 2^{\beta}\e_3, 2^{\gamma}\e_4\rangle$ for $\e_1,\e_2\in \z_2^{\times}$ and integers $\beta, \gamma$ such that $\gamma \ge \beta \ge 2$, then $L$ is strongly $s$-regular if and only if $\lambda_2(L)$ is strongly $s$-regular. Furthermore, if one of them is true, then $m_s(L) = 2\cdot m_s(\lambda_2(L))$.

\item [(ii)] If $0 \le \alpha \le 1$ and $M$ is an improper modular lattice with norm contained in $4\z_2$ or $M \simeq \langle 2^{\beta}\e_3, 2^{\gamma}\e_4\rangle$ for $\e_1,\e_2\in \z_2^{\times}$ and integers $\beta, \gamma$ such that $0 \le \alpha \le \beta  \le \gamma$, then $L$ is strongly $s$-regular if and only if $\lambda_2^2(L)$ is strongly $s$-regular. Furthermore, if one of them is true, then $m_s(L) = 2\cdot m_s(\lambda_2^2(L))$.
\end{itemize}

\end{prop}
\begin{proof}
The proof is quite similar to the odd case. 
\end{proof}

\begin{thm}\label{terminal}
Let $L$ be a strongly $s$-regular quaternary $\z$-lattice. Then there is a positive integer $N$ such that
\begin{enumerate} 
\item $\lambda_N(L)$ is a strongly $s$-regular  lattice such that $m_s(\lambda_N(L))$ is odd square free;
\item for any prime $p$ dividing $m_s(\lambda_N(L))$,  
\begin{equation}\label{except}
\begin{array}{ll}
\lambda_N(L)_p \simeq &\langle \Delta_p, p^{2}\e_1, p^{2}\e_2, p^{2}\e_3 \rangle \quad \text{or} \quad  \langle \Delta_p, p, -p, p^{\gamma}\e_3 \rangle \\

                      &\quad \text{or} \quad \langle \Delta_p, p\e_1, p\e_2, p\e_3 \rangle,
\end{array}
\end{equation}
where $\e_1, \e_2, \e_3 \in \z_p^{\times}$ and $\gamma \ge 2$.
\end{enumerate}          
\end{thm}
\begin{proof}
The theorem is the direct consequence of Proposition \ref{odd} and \ref{even}.
\end{proof}

\begin{defn}
Let $L$ be a strongly $s$-regular quaternary $\z$-lattice. We say $L$ is {\it terminal} if
$m_s(L)$ is odd square free and $L_p$ isometric to the one of three $\z_p$-lattices in \eqref{except} for any prime $p$ dividing $m_s(L)$.
\end{defn}

%

\begin{lem}\label{ternary}
Let $L$ be any ternary $\z$-lattice and let $m=q_1q_2\cdots q_s$ be an odd square free integer. Then there is a positive integer $N$ such that 
$$
r(m^2,L) \le N.
$$
Furthermore, if $p$ does not divide $dL$, then for any integer $t\ge 1$, we have
$$
r(p^{2t}m^2,L) \le (p+2)(p+1)^{t-1}N.
$$
In particular, if $m=1$, then it is well known that $r(1,L)\le 12$.
\end{lem}
\begin{proof}
Let $\{x_1,x_2,x_3\}$ be a Minkowski reduced basis for $L$ such that
$$
(B(x_i,x_j))\simeq \begin{pmatrix} a&f&e\\f&b&d\\e&d&c \end{pmatrix}.
$$
Note that  $0\le a\le b\le c$ and $2\vert f \vert \le a,\ 2\vert e \vert \le a,\ 2\vert d \vert \le b $.

Now assume that $c \le m^2$. This implies that the discriminant of $L$ is bounded by a constant depending only on $m$. Hence there are only finitely many ternary $\z$-lattices $L$ such that $c\le m^2$. We put
$$
N'=\text{max}\{r(m^2,L) \mid \text{$L$ is a ternary $\z$-lattice such that $c\le m^2$} \}.
$$

Next assume that $c>m^2$. Then we have
$$
r(m^2,L)=r\left(m^2, \begin{pmatrix} a&f\\f&b \end{pmatrix}\right) \le 6\cdot 3^{s}.
$$
Take $N=\text{max}\{N',6\cdot 3^s\}$. Then we have
\begin{equation}\label{ternary1}
r(m^2,L)\le N.
\end{equation}
On the other hand, if $p$ does not divide $dL$, then the action of Hecke operators $T(p^2)$ on theta series of the lattice $L$ gives 
$$
r(p^2m^2,L)+\left(\frac{-2m^2dL}{p}\right)r(m^2,L)+p\cdot r\left(\frac{m^2}{p^2},L\right)=\sum_{[L']\in \gen(L)} \frac{r^*(pL',L)}{o(L')}r(m^2,L'),
$$
where $r^*(pL',L)$ is the number of primitive representations of $pL'$ by $L$. Here, if $p^2 \nmid m^2$, then $r\left(\frac{m^2}{p^2},L\right)=0$. It is well
known that 
$$
\sum_{[L']\in \gen(L)} \frac{r^*(pL',L)}{o(L')}=p+1.
$$
For details, see Chapter 3 of  \cite {an}. Hence by \eqref{ternary1}, we have
\begin{equation}\label{hecke}
\begin{array}{ll}
\!\!r(p^2m^2,L) \!\!\!\!\!&\displaystyle = \sum_{[L']\in \gen(L)} \!\frac{r^*(pL',L)}{o(L')}r(m^2,L')\!-\!\left(\frac{-2m^2dL}{p}\right)r(m^2,L) -  pr\left(\frac{m^2}{p^2},L\right) \\

            \!\!\!&\displaystyle \le \left(p+1-\left(\frac{-2m^2dL}{p}\right)\right)N \le (p+2)N.
       
\end{array}
\end{equation}
Similarly, by \eqref{hecke}, we have

$$
\begin{array}{ll}
r(p^4m^2,L)\displaystyle &\displaystyle= \sum_{[L']\in \gen(L)} \frac{r^*(pL',L)}{o(L')}r(p^2m^2,L')-p\cdot r(m^2,L)\\
                         &\displaystyle  \le (p+1)(p+2)N.

\end{array}
$$
By repeating the same argument given above, we finally have
$$
r(p^{2t}m^2,L) \le (p+2)(p+1)^{t-1}N,
$$
for any integer $t\ge 2$.
\end{proof}

\begin{thm}\label{main}
For any positive integer $m$, there are only finitely many strongly $s$-regular quaternary $\z$-lattices $L$ up to isometry such that $m_s(L)=m$. 
\end{thm}
\begin{proof}
Note that for any prime $p$ and for any quaternary $\z$-lattice $K$, there are finitely many $\z$-lattices whose $\lambda_p$-transformation is isometric to $K$. Hence by Lemma \ref{terminal}, it is enough to show that there are finitely many terminal strongly $s$-regular quaternary $\z$-lattice $L$ such that $m_s(L)=m$ under the assumption that $m=q_1q_2\cdots q_s$ is an odd square free integer. If $m=1$, then $s=0$.

Let $\{x_1,x_2,x_3,x_4\}$ be a Minkowski reduced basis for $L$ whose Gram matrix is given by
$$
(B(x_i,x_j))\simeq \begin{pmatrix} a&l&k&h\\l&b&g&f\\k&g&c&e\\h&f&e&d \end{pmatrix}.
$$
Let $p_t$ be the $t$-th smallest odd prime so that $p_1=3,\ p_2=5$ and so on. Put
$$
t'=\min\{t\in \n \mid 16m^8p_{t}^6 < p_1p_2\cdots p_{t-1}\}.
$$
Since $p_t< 2p_{t-1}$ for any positive integer $t$ by the Bertrand-Chebyshev Theorem, such an integer always exists. We define a ternary $\z$-lattice $L'$ by 
$$
L' \simeq  \begin{pmatrix} a&l&k\\l&b&g\\k&g&c \end{pmatrix}.
$$ 
Note that since $L$ is the quaternary $\z$-lattice and the basis $\{x_1,x_2,x_3,x_4\}$ is a Minkowski reduced basis for $L$, we have $r(n,L)=r(n,L')$ for any positive integer $n< d$.
Then by Lemma \ref{ternary}, there is a positive integer $N$ such that
$$
r\left(m^2,L'\right) \le N.
$$
Let $t''$ be the smallest positive integer such that $2(p_{t''}^2-p_{t''}+1) > N(p_{t''}+2) $ and let $t_0=\max\{t', t''\}$. Let $t_1$ be the positive integer such that $p_1p_2\cdots p_{t_1-1} \mid dL$, but $p_{t_1} \nmid dL$.

First, assume that $t_1\ge t_0$. Then we have
$$
16m^8p_{t_1}^6 < p_1p_2\cdots p_{t_1-1} <dL \le 16abcd \le 16ad^3 \le 16m^2d^3.
$$
Hence $m^2p_{t_1}^2 <d$. If $p_{t_1}$ does not divide $dL'$, then by Lemma \ref{ternary}, we have 
$$
r(m^2p_{t_1}^2,L)=r\left(m^2p_{t_1}^2,L'\right) \le (p_t+2)N.
$$
If $p_{t_1}$ divides $dL'$, then by Lemma \ref{aniso}, \ref{iso} and \ref{ternary}, we have
$$
r(m^2p_{t_1}^2,L)=r\left(m^2p_{t_1}^2,L'\right) \le 4N \le (p_t+2)N.
$$
However, since $L$ is strongly $s$-regular and $p_{t_1}\nmid 16mdL$, we have
$$
r(m^2p_{t_1}^2,L)=r(m^2,L)\left(p_{t_1}^2+ \left(\frac{dL}{p_{t_1}}\right)p_{t_1}+1\right)\ge 2(p_{t_1}^2-p_{t_1}+1) > N(p_{t_1}+2).
$$
This is a contradiction.

Finally, assume that $t_1 < t_0$. Now choose a positive integer $\mu_0$ such that 
$$
p_{t_1}^{2\mu_0}-1 > (p_{t_1}+2)(p_{t_1}+1)^{\mu_0-1}N.
$$
If $d>m^2p_{t_1}^{2\mu_0}$, then by Lemma \ref{ternary},
$$
r(m^2p_{t_1}^{2\mu_0},L)=r(m^2p_{t_1}^{2\mu_0},L') \le (p_{t_1}+2)(p_{t_1}+1)^{\mu_0-1}N.
$$
However, since $L$ is strongly $s$-regular, we have
$$
\begin{array}{ll}
r(m^2p_{t_1}^{2\mu_0},L)& \displaystyle= r(m^2,L)\cdot \left( \frac {p_{t_1}^{2\mu_0+1}-\left(\frac{dL}{p_{t_1}}\right)}{p_{t_1}-\left(\frac{dL}{p_{t_1}}\right)}\right)\ge 2\cdot \frac12(p_{t_1}^{2\mu_0}-1)=p_{t_1}^{2\mu_0}-1\\
&\displaystyle>(p_{t_1}+2)(p_{t_1}+1)^{\mu_0-1}N,
\end{array}
$$
which is a contradiction. Hence $d \le m^2p_{t_1}^{2\mu_0}$. Therefore the discriminant of $L$ is bounded by a constant depending only on $m$. This completes the proof.
\end{proof}

Let $L$ be a quaternary $\z$-lattice such that $r(n^2,L)=r(n^2,\gen(L))$ for any integer $n$. For any positive integer $n$, let $n_1$ and $n_2$ be positive integers such that $P(n_1)\subset P(2dL)$, $(n_2,2dL)=1$ and $n=n_1n_2$. Here  $P(n)$ denotes the set of prime factors of $n$. Then by Lemma \ref {Min}, we have
$$
r(n_1^2n_2^2,L)=r(n_1^2,L)\cdot \prod_{p\nmid 2dL}h_p(dL,\mu_p),
$$
where $\mu_p=\text{ord}_p(n)$ for any prime $p$ and
$$
h_p(dL,\mu_p)=\sum_{t=0}^{2\mu_p}\left(\frac{dL}p\right)^tp^{t}.
$$
Hence, the $\z$-lattice $L$ is strongly $s$-regular. Therefore, by Theorem \ref{main}, we have the following:
\begin{cor}
For any positive integer $m$, there are only finitely many quaternary $\z$-lattices $L$ up to isometry such that $r(n^2,L)=r(n^2,\gen(L))$ for any integer $n$ and  $m_s(L)=m$. 
\end{cor}

\section{Strongly $s$-regular diagonal quaternary lattices representing one}
In this section, we will find all strongly $s$-regular diagonal quaternary lattice $L$ with $m_s(L)=1$. To do this, we need the following lemma.

\begin{lem}\label {357}
Let $L=\langle 1,a,b,c\rangle$ be a strongly $s$-regular diagonal quaternary $\z$-lattice with $m_s(L)=1$ and $1\le a\le b\le c$. Then $dL$ is not divisible by at least one prime in $\{3,5,7\}$.
\end{lem}
\begin{proof}
Let $L=\langle 1,a,b,c\rangle$ be a strongly $s$-regular diagonal quaternary $\z$-lattice with $m_s(L)=1$ and $1\le a\le b\le c$. Let $p_t$ be the $t$-th smallest odd prime. Suppose that, on the contrary, $p_1p_2\cdots p_{t-1} \mid dL$, but $p_t \nmid dL$ for some $t\ge 4$.

First, assume that $t=4$. Since $11\nmid dL=16abc$, we have
\begin{equation}\label{t=4}
r(11^2,L)=r(1,L)\left(11^2+\left(\frac{dL}{11}\right)11+1\right)\ge 2\cdot 111=222.
\end{equation}
If $c\ge 11^2+1$, then $r(11^2,L)=r(11^2,\langle1,a,b\rangle)\le 12(11+2)=156$ by Lemma \ref{ternary}. This is a contradiction. Hence we have $1\le a\le b\le c\le 121$. For all possible finite cases, one may check that there are no quaternary $\z$-lattices such that the equation \eqref{t=4} holds. The cases when $5\le t \le 9$ can be dealt with similar manner to this.

Finally, assume that $t\ge 10$. Then $p_t^6< p_1p_2\cdots p_{t-1}\le abc \le c^3$ by the Bertrand-Chebyshev Theorem. 
Hence $p_t^2 \le c$ and by Lemma \ref{ternary}, we have
$$
r(p_t^2, L)=r(p_t^2,\langle 1,a,b\rangle)\le 12(p_t+2).
$$
However, since $L$ is strongly $s$-regular and $p_t \nmid dL$, we have
$$
r(p_t^2,L)=r(1,L)\left( p_t^2+\left(\frac{dL}{p_t}\right)p_t+1\right) \ge 2(p_t^2-p_t+1).
$$
Since $p_t\ge 31$, we have $12(p_t+2) < 2(p_t^2-p_t+1)$. This is a contradiction.
\end{proof}

\begin{table}[t]
\caption{Strongly $s$-regular lattices $L$}
\begin{tabular}{|c|l|}\hline
 &\hskip 4.5cm $L$ \\\hline
\multirow{4}{*}{$3\nmid dL \ (18)$}  & $\langle1,1,1,1\rangle$, $\langle1,1,1,2\rangle$, $\langle1,1,1,4\rangle$, $\langle1,1,1,5\rangle$, $\langle1,1,1,8\rangle$,  \\
 &$\langle1,1,2,2\rangle$, $\langle1,1,2,4\rangle$, $\langle1,1,4,4\rangle$, $\langle1,1,4,8\rangle$, $\langle1,2,2,2\rangle$,  \\
&$\langle1,2,2,4\rangle$, $\langle1,2,2,8\rangle$, $\langle1,2,4,4\rangle$, $\langle1,2,8,8\rangle$, $\langle1,4,4,4\rangle$,  \\
    &$\langle1,4,4,8\rangle$, $\langle1,5,5,5\rangle$, $\langle1,8,8,8\rangle$ \\
\hline
   
\multirow{3}{*}{$3\mid dL$, $5\nmid dL \ (14)$}\!\! & $\langle1,1,1,3\rangle$, $\langle1,1,2,3\rangle$, $\langle1,1,2,6\rangle$, $\langle1,1,3,3\rangle$, $\langle1,1,3,9\rangle$,  \\
   &$\langle1,2,2,3\rangle$, $\langle1,2,2,6\rangle$, $\langle1,2,4,6\rangle$, $L_1=\langle\textbf{1,2,6,16}\rangle$, $\langle1,3,3,3\rangle$,\!\\
    &$\langle1,3,3,6\rangle$, $\langle1,3,3,9\rangle$, $\langle1,3,6,6\rangle$, $\langle1,3,9,9\rangle$ \\\hline
 
\multirow{1}{*}{$15\mid dL$, $7\nmid dL\ (2)$} \!\!\!&  $L_2=\langle\textbf{1,1,3,5}\rangle$, $L_3=\langle\textbf{1,2,3,10}\rangle$ \\
\hline                                         
\end{tabular}
\end{table}

\begin{thm}\label{main2}
There are exactly $34$ strongly $s$-regular diagonal quaternary $\z$-lattices $L$ up to isometry such that $m_s(L)=1$, which are listed in Table $1$.
\end{thm}
\begin{proof}
Note that all diagonal quaternary lattices except $L_1$, $L_2$ and $L_3$ highlighted in boldface in Table $1$ have class number $1$. By Proposition \ref{indis}, they are strongly $s$-regular. There are exactly $3$ strongly $s$-regular diagonal quaternary $\z$-lattice with class number $2$ in Table $1$. The proof of the strongly $s$-regularities of these lattices will be given in Section $5$.

Let $L=\langle 1,a,b,c\rangle$ be a strongly $s$-regular diagonal quaternary $\z$-lattice. By Lemma \ref{357}, the discriminant of $L$, which is $16abc$, is not divisible by at least one prime in $\{3,5,7\}$.

First, assume that $dL$ is not divisible by $3$. Assume that $r(1,L)=2$, that is, $a\ge 2$.
Then we have $r(9,L)=14$ or $26$. If $c>9$, then $r(9,L)=r(9,\langle 1,a,b\rangle)=14$ or $26$. In this case, there does not exist a strongly $s$-regular diagonal lattice $L$ such that $3 \nmid dL$. Hence we have $a\le b\le c \le 9$. For all possible cases, $L$ is strongly $s$-regular if and only if the class number of $L$ is one. Assume that $a=1$ and $b\ge 2$. Then we have $r(9,L)=28$ or $52$.
If $c>9$, then by Lemma \ref{ternary}, we have $r(9,L)\le 5\cdot 4=20$. This is a contradiction. 
Hence we have $2\le b\le c\le 9$. In this case, there are exactly $4$ strongly $s$-regular diagonal lattices with class number $1$. Assume that $a=b=1$. Then we have $c=1,2,4,5$ or $8$.

Next, assume that $dL$ is divisible by $3$, but is not divisible by $5$. Assume that $a\ge 2$. Then we have $r(25,L)=42$ or $62$. This implies that $2 \le a \le  b\le c\le 9$. For all possible cases, one may check that $L$ is isometric to one of
$$
\begin{array}{ll}
&\langle1,2,2,3\rangle,\ \langle1,2,2,6\rangle,\ \langle1,2,3,3\rangle^{\dagger},\ \langle1,2,4,6\rangle,\ L_1=\langle1,2,6,16\rangle^{\dagger},\ \langle1,3,3,3\rangle,\\
&\langle1,3,3,6\rangle,\ \langle1,3,3,9\rangle,\ \langle1,3,3,18\rangle^{\dagger},\ \langle1,3,6,6\rangle\ \text{and}\ \langle1,3,9,9\rangle.
\end{array}
$$
The lattices with dagger mark have class number $2$. By definition of strongly $s$-regularity, the lattices $\langle1,2,3,3\rangle^{\dagger}$ and $\langle1,3,3,18\rangle^{\dagger}$ are not strongly $s$-regular. The strongly $s$-regularity of $L_1$ will be proved in Proposition \ref{111}.
Assume $a=1$. In this case, there are exactly $5$ strongly $s$-regular diagonal lattices with class number $1$.
 
Finally, assume that $15\mid dL$ and $7 \nmid dL$. In this case, $L$ is isometric to one of
$$
L_2=\langle 1,1,3,5 \rangle \quad \text{and}\quad L_3=\langle 1,2,3,10 \rangle.
$$ 
Note that $L_2$ and $L_3$ have class number $2$. The strongly $s$-regularities of $L_2$ and $L_3$
will be proved in Proposition \ref{222} and Proposition \ref{333}, respectively.
\end{proof}

\section{Nontrivial strongly $s$-regular diagonal quaternary lattices}
In this section, we will prove the strongly $s$-regularities of diagonal quaternary lattices $L_1$, $L_2$ and $L_3$ in Table $1$, which are of class number $2$.

\begin{prop}\label{111}
The genus $\gen(L_1)$ is indistinguishable by squares. In particular, the quaternary $\z$-lattice $L_1$ is strongly $s$-regular.
\end{prop}
\begin{proof}
Note that $h(L_1)=2$ and $\gen(L_1)=\{[L_1],[L_1']\}$, where
$$
L_1'=\langle1\rangle \perp \begin{pmatrix}6&2&-2\\2&6&2\\-2&2&8 \end{pmatrix}.
$$
Let $\{x_1,x_2,x_3,x_4\}$ be the basis for $L_1$ whose Gram matrix is given in Table $1$.

Assume that $Q(ax_1+bx_2+cx_3+dx_4)=4n$ for any nonnegative integer $n$. Then $a^2+2b^2+6c^2+16d^2=4n$ and $a\equiv 0 \pmod 2$. This also implies that $b-c\equiv0 \pmod 2$. Hence we have
$$
r(4n,L_1)=r(4n,\z(2x_1)+\z (2x_2)+\z (x_2+x_3)+\z x_4),
$$
which implies that 
$$
r(4n,L_1)=r\left(4n,\langle4\rangle \perp \begin{pmatrix} 8&4\\4&8 \end{pmatrix}\perp\langle16\rangle\right).
$$
Similarly, we have
$$
r(4n,L_1')=r\left(4n,\langle4\rangle \perp \begin{pmatrix} 8&4\\4&8 \end{pmatrix}\perp\langle16\rangle\right).
$$
Therefore, we have 
\begin{equation}\label{l1}
r(4n,L_1)=r(4n,L_1').
\end{equation}

On the other hand, assume that $Q(ax_1+bx_2+cx_3+dx_4)=4n+1$ for any nonnegative integer $n$. Then $a^2+2b^2+6c^2+16d^2=4n+1$ and $b-c\equiv 0 \pmod 2$. Hence we have
$$
r(4n+1,L_1)=r(4n+1,\z x_1+\z (2x_2)+\z (x_2+x_3)+\z x_4),
$$
which implies that
$$
r(4n+1,L_1)=r\left(4n+1,\langle1\rangle \perp \begin{pmatrix} 8&4\\4&8 \end{pmatrix}\perp\langle16\rangle\right).
$$
Similarly, we have
$$
r(4n+1,L_1')=r\left(4n+1,\langle1\rangle \perp \begin{pmatrix} 8&4\\4&8 \end{pmatrix}\perp\langle16\rangle\right).
$$
Therefore, we have 
\begin{equation}\label{l11}
r(4n+1,L_1)=r(4n+1,L_1').
\end{equation}
By \eqref{l1} and \eqref{l11}, the genus $\gen(L_1)$ is indistinguishable by squares and by Proposition \ref{indis}, $L_1$ and $L_1'$ are strongly $s$-regular.
\end{proof}

\begin{prop}\label{222}
The genus $\gen(L_2)$ is indistinguishable by squares. In particular, quaternary $\z$-lattice $L_2$ is strongly $s$-regular.
\end{prop}
\begin{proof}
Note that $h(L_2)=2$ and $\gen(L_2)=\{[L_2], [L_2']\}$. Here,
$$
L_2'=\langle1,1\rangle \perp \begin{pmatrix}2&1\\1&8 \end{pmatrix}.
$$
Let $\{x_1,x_2,x_3,x_4\}$ be the basis for $L_2$ whose Gram matrix is given in Table $1$. 

Assume that $Q(ax_1+bx_2+cx_3+dx_4)=3n+1$ for any nonnegative integer $n$. Then $a^2+b^2+3c^2+5d^2=3n+1$ and by a direct computation, we have
$$
(a,b,c)\equiv (0,\pm1,0), (\pm1,0,0) \ \text{or} \ (\pm1,\pm1,\pm1) \pmod 3.
$$
Hence we may see that
$$
\begin{array}{ll}
r(3n+1,L_2)=\displaystyle \frac12(\!\!\!\!&r(3n+1,\z x_1+\z(3x_2)+\z x_3+\z(x_2+x_4))\\
                                \!\!\!\!&r(3n+1,\z x_1+\z(3x_2)+\z x_3+\z(-x_2+x_4))\\
                                \!\!\!\!&r(3n+1,\z (3x_1)+\z(x_2+\z x_3+\z(x_1+x_4))\\
                                \!\!\!\!&r(3n+1,\z (3x_1)+\z x_2+\z x_3+\z(-x_2+x_4))\ ).
\end{array}
$$
This implies that
$$
r(3n+1,L_2)=\frac12(4 r(3n+1,K))=2r(3n+1,K).
$$
Here, $K=\langle1,3\rangle \perp \begin{pmatrix}6&3\\3&9\end{pmatrix}$.
Similarly, we also have $r(3n+1,L_2')=2r(3n+1,K)$. Therefore, we have
\begin{equation}\label{eq1}
r(3n+1,L_2)=r(3n+1,L_2'),
\end{equation}
for any nonnegative integer $n$. 

On the other hand, assume that $Q(ax_1+bx_2+cx_3+dx_4)=(3n)^2=9n^2$ for any nonnegative integer $n$.
Then $a\equiv 0\pmod 3$ or $b\equiv 0\pmod 3$. Hence we have
$$
\begin{array}{ll}
\!\!r(9n^2, L_2)\!\!\!\!\!&=\!r(9n^2,\z (3x_1)\!+\z x_2\!+\z x_3+\z x_4)\!+\!r(9n^2,\z x_1\!+\z (3x_2)\!+\z x_3\!+\z x_4)\\
                &\hspace{0.4cm}-r(9n^2,\z (3x_1)+\z (3x_2)+\z x_3+\z x_4),
\end{array}
$$
which implies that 
$$
\begin{array}{ll}
r(9n^2, L_2)\!\!\!&=2r(9n^2,\langle1,3,5,9\rangle)-r(9n^2,\langle3,5,9,9 \rangle)\\
          &=2r(9n^2,\langle1,3,5,9\rangle)-r(n^2,L_2).
\end{array}
$$
One may show that (see, for example, Lemma \ref{iso})
$$
\begin{array}{ll}
r(9n^2,\langle1,3,5,9\rangle)\!\!\!&=2r(9n^2,T)-r(9n^2,\langle3,9,9,45\rangle)\\
                             &=2r(9n^2,T)-r(n^2,L_2).   
\end{array}
$$
Here, $T=\langle3\rangle \perp \begin{pmatrix}6&3\\3&9 \end{pmatrix}\perp\langle9\rangle  $. Therefore we have
\begin{equation}\label{eq2}
r(9n^2,L_2)=4r(9n^2,T)-3r(n^2,L_2).
\end{equation}
Similarly, we also have 
\begin{equation}\label{eq3}
r(9n^2,L_2')=4r(9n^2,T)-3r(n^2,L_2').
\end{equation}
By \eqref{eq1}, \eqref{eq2} and \eqref{eq3}, we see that
$$
r(n^2,L_2)=r(n^2,L_2'),
$$
for any nonnegative integer $n$. Therefore, the genus $\gen(L_2)$ is indistinguishable by squares and $L_2$ and $L_2'$ are strongly $s$-regular.
\end{proof}

From now on, we prove the strongly $s$-regularity of the lattice $L_3$ in Table $1$. To deal with this, we need some results from the theory of modular forms. For some relations between representations of quadratic forms and modular forms, see Chapter $10$ of \cite{wp}.

Let $N$ be a positive integer and let $\Gamma_0(N)$ be the Hecke congruence subgroup of $\text{SL}_2(\z)$. We denote the space of cusp forms weight $k$ with character $\chi$ for $\Gamma_0(N)$ by $S_k(N,\chi)$.

We define the {\it Dedekind's eta-function} $\eta(z)$ by
$$
\eta(z)=q^{1/24}\prod_{n=1}^{\infty}(1-q^{n}) \quad (q=e^{2\pi iz}).
$$
An {\it eta-quotient} $f(z)$ is defined to be a finite product of the form
$$
f(z)=\prod_{\delta\mid N}\eta(\delta z)^{r_{\delta}},
$$
where $N$ is a positive integer and each $r_{\delta}$ is an integer.

\begin{thm}[\cite{ne1,ne2}]\label{modularform}
If $f(z)=\prod_{\delta\mid N}\eta(\delta z)^{r_{\delta}}$ is an eta-quotient with $k=\frac{1}{2}\sum_{\delta\mid N}r_{\delta}\in\mathbb{Z}$, with additional properties that
$$\sum_{\delta\mid N}\delta r_{\delta}\equiv 0\pmod{24}$$
and
$$\sum_{\delta\mid N}\frac{N}{\delta}r_{\delta}\equiv 0\pmod{24},$$
then $f(z)$ satisfies
$$f\left(\frac{az+b}{cz+d}\right)=\chi(d)(cz+d)^{k}f(z)$$
for every $\left(\begin{smallmatrix}a & b \\ c & d\end{smallmatrix}\right)\in\Gamma_{0}(N)$. Here the character $\chi$ is defined by $\chi(d):=\left(\frac{(-1)^{k}s}{d}\right)$, where $s:=\prod_{\delta\mid N}\delta^{|r_{\delta}|}$.
\end{thm}

\begin{thm}[\cite{li}]\label{cusp}
Let $c,d$ and $N$ be positive integers with $d\mid N$ and $\gcd{(c,d)}=1$. If $f(z)=\prod_{\delta\mid N}\eta(\delta z)^{r_{\delta}}$ is an eta-quotient, then the order of vanishing of $f(z)$ at the cusp $c/d$ is
$$\frac{N}{24}\sum_{\delta\mid N}\frac{\gcd{(d,\delta)}^{2}r_{\delta}}{\gcd{(d,N/d)}d\delta}.$$
\end{thm}

\begin{lem}\label{120}
The eta-quotients
$$
\begin{array}{ll}
\displaystyle F_{120,1}(z)\!\!\!&\displaystyle:=\frac{\eta(2z)^2\eta(15z)^3}{\eta(z)}=q^2+q^3+q^5+q^8+q^{12}-2q^{17}-3q^{18}+\cdots,\\[0.35cm]

\displaystyle F_{120,2}(z)\!\!\!&\displaystyle:=\frac{\eta(2z)\eta(10z)^3\eta(30z)^2}{\eta(5z)\eta(15z)}=q^3-q^5-q^7+q^8-q^{10}-q^{12}+q^{15}+\cdots,\\[0.35cm]

\displaystyle F_{120,3}(z)\!\!\!&\displaystyle:=\frac{\eta(2z)^2\eta(5z)\eta(60z)^3}{\eta(z)\eta(20z)}=q^7+q^8+q^{10}-q^{12}-q^{15}-2q^{18}-q^{20}+\cdots,
\end{array}
$$
are in $S_2\left(120,\left(\frac{60}{\cdot}\right)\right)$. Furthermore, if $F_{120,i}(z)=\sum_{n=1}^{\infty}A_{i}(n)q^n$ for each $i=1,2,3$, then we have
$$
\begin{array}{ll}
A_1(n)&=\displaystyle \frac14\sum_{\substack{a,b\in \z\\a^2+15b^2=8n}}\left(\frac4a\right)\left(\frac{-4}{b}\right)bq^n, \\[0.8cm]

A_2(n)&=\displaystyle \frac1{16}\sum_{\substack{a,b,c,d\in \z\\2a^2+10b^2+15c^2+45d^2=24n}}\left(\frac{12}{ab}\right)\left(\frac{4}{cd}\right)q^n,\\[0.8cm]

A_3(n)&=\displaystyle \frac14\sum_{\substack{a,b\in \z, c\in \n, 3\nmid c\\3a^2+5b^2+160c=24n}}\left(\frac4a\right)\left(\frac{12}{b}\right)\left(\sum_{d\mid n}\left(\frac{d}{3}\right)\right)bq^n.
\end{array}
$$
In particular, if $n\equiv 1$ or $4\pmod 5$, then $A_{i}(n)=0$ for each $i=1,2,3$.
\end{lem}
\begin{proof}
By Theorem \ref{modularform} and \ref{cusp}, we may easily check that $F_{120,i}(z)\in S_2\left(120,\left(\frac{60}{\cdot}\right)\right)$ is the weight 2 cusp form for each $i=1,2,3$. It is well known (see, for example, Chapter $5$ of \cite{cs}) that
\begin{equation}\label{eta}
\begin{array}{ll}
\displaystyle \eta(z)&=\displaystyle\frac12\sum_{n\in \z}\left(\frac{12}{n}\right)q^{n^2/24},\quad \frac{\eta(2z)^2}{\eta(z)}=\frac12\sum_{n\in \z}\left(\frac{4}{n}\right)q^{n^2/8},\\

\displaystyle\eta(z)^3&=\displaystyle\frac12\sum_{n\in \z}\left(\frac{-4}{n}\right)nq^{n^2/8}.
\end{array}
\end{equation}
By Example 11.4 of \cite{kh}, we see that
\begin{equation}\label{spe}
\frac{\eta(3z)^3}{\eta(z)}=\sum_{n\in\n, 3\nmid n}\left(\sum_{d\mid n}\left(\frac{d}{3}\right) \right)q^{n/3}.
\end{equation}
Since
$$
\begin{array}{ll}
F_{120,1}(z)&=\displaystyle\frac{\eta(2z)^2}{\eta(z)}\cdot \eta(15z)^3, \quad F_{120,2}(z)=\eta(2z)\cdot\eta(10z)\cdot\frac{\eta(10z)^2}{\eta(5z)}\cdot\frac{\eta(30z)^2}{\eta(15z)},\\[0.35cm]

F_{120,3}(z)&=\displaystyle\frac{\eta(2z)^2}{\eta(z)}\cdot\eta(5z)\cdot\frac{\eta(60z)^3}{\eta(20z)},
\end{array}
$$
by using \eqref{eta} and \eqref{spe}, we have the formulas for $A_1$, $A_2$ and $A_3$. Then it is obvious that if $n\equiv 1$ or $4\pmod 5$, then $A_{i}(n)=0$ for each $i=1,2,3$.
\end{proof}

\begin{prop}\label{333}
The genus $\gen(L_3)$ is indistinguishable by squares. In particular, the quaternary $\z$-lattice $L_3$ is strongly $s$-regular.
\end{prop}
\begin{proof}
Note that $h(L_3)=2$ and $\gen(L_3)=\{[L_3], [L_3']\}$. Here,
$$
L_3'=\langle1\rangle \perp \begin{pmatrix}3&-1&1\\-1&5&1\\1&1&5\end{pmatrix}.
$$

First, we prove that if $n\equiv 1$ or $4 \pmod 5$, then $r(n,L_3)=r(n,L_3')$.
We let
$$
\phi(z)=\frac12\sum_{n=1}^{\infty}(r(n,L_3)-r(n,L_3'))q^{n}.
$$
Then it is known that $\phi(z)\in S_2\left(120,\left(\frac{60}{\cdot}\right)\right)$ is the weight $2$ cusp form. By the Sturm's bound, a modular form of weight $2$ for $\Gamma_0(120)$ is uniquely determined by the first $\frac 16 [\text{SL}_2(\z): \Gamma_0(120)]$ Fourier coefficients. Further by $[\text{SL}_2(\z): \Gamma_0(N)]= N \prod_{p \mid N}(1+\frac 1p)$, we have
$\frac 16 [\text{SL}_2(\z): \Gamma_0(120)]=48.$ One may easily check that the first $48$ Fourier coefficients of $\phi(z)$ are equal to those of 
$$
F_{120,1}(z)+F_{120,2}(z)-4F_{120,3}(z).
$$
Here, $F_{120,1}(z)$, $F_{120,2}(z)$ and $F_{120,3}(z)$ are defined in Lemma \ref{120}.
Therefore, we have $$\phi(z)=F_{120,1}(z)+F_{120,2}(z)-4F_{120,3}(z).$$ This implies that
$$
\frac12(r(n,L_3)-r(n,L_3'))=A_1(n)+A_2(n)-4A_3(n).
$$
By Lemma \ref{120}, if $n\equiv 1$ or $4 \pmod 5$, then we have
\begin{equation}\label{145}
r(n,L_3)=r(n,L_3').
\end{equation}

Next, let $\{x_1,x_2,x_3,x_4\}$ be the basis for $L_3$ whose Gram matrix is given in Table $1$. 
Assume that $Q(ax_1+bx_2+cx_3+dx_4)=(5n)^2=25n^2$. Then $a^2+2b^2+3c^2+10d^2=25n^2$. By a direct computation, we have
$$
a+b\equiv 0, \quad a-b\equiv 0, \quad b+c\equiv 0  \quad \text{or} \quad b-c\equiv 0 \pmod 5.
$$ 
Since the case when $a \equiv b\equiv c\equiv 0\pmod 5$ occurs in all of the above cases, we have
$$
\begin{array}{ll}
r(25n^2,L_3)=&r(25n^2,\z(5x_1)+\z(-x_1+x_2)+\z x_3+ \z x_4)\\
             &+r(25n^2,\z(5x_1)+\z(x_1+x_2)+\z x_3+ \z x_4)\\
             &+r(25n^2,\z x_1)+\z(5x_2)+\z (-x_2+x_3)+ \z x_4)\\
             &+r(25n^2,\z x_1)+\z(5x_2)+\z (x_2+x_3)+ \z x_4)\\
             &-3r(25n^2,\z(5x_1)+\z(5x_2)+\z (5x_3)+ \z x_4),
\end{array}
$$
which implies that 
\begin{equation}\label{e1}
r(25n^2,L_3)=2r(5n^2,M)+2r(25n^2,K_1)-3r(n^2,L_3),
\end{equation}
Here, $M=\langle1,2,5,6\rangle$ and $K_1=\langle3\rangle\perp\begin{pmatrix}3&0&5\\0&10&0\\5&0&25\end{pmatrix}$. Similarly, we have
\begin{equation}\label{e2}
r(25n^2,K_1)=2r(5n^2,N)-r(n^2,L_3),
\end{equation}
where $N=\langle2\rangle\perp\begin{pmatrix}3&0&1\\0&3&1\\1&1&4\end{pmatrix}$. Hence, by \eqref{e1} and \eqref{e2}, we have
\begin{equation}\label{e3}
r(25n^2,L_3)=2r(5n^2,M)+4r(5n^2,N)-5r(n^2,L_3).
\end{equation}

Finally, let $\{y_1,y_2,y_3,y_4\}$ be the basis for $L_3'$ whose Gram matrix is given above. Assume that $Q(ay_1+by_2+cy_3+dy_4)=25n^2$ for any nonnegative integer $n$. Then $a^2+3b^2+5c^2+5d^2-2bc+2bd+2cd=25n^2$. By a direct computation, we have
$$
b+c\equiv 0,\quad b-d\equiv 0 ,\quad b+2c+d\equiv 0 \quad\text{or}\quad b-c-2d\equiv 0 \pmod 5.
$$ 
Since, in all of the above cases, the case when $a\equiv 0 \pmod 5$ and $c+d\equiv 0 \pmod 5$ occurs, we have
$$
\begin{array}{ll}
r(25n^2,L_3')=&r(25n^2,\z y_1+\z(5y_2)+\z (-y_2+y_3)+ \z y_4)\\
             &+r(25n^2,\z y_1+\z(5y_2)+\z y_3+ \z (y_2+y_4))\\
             &+r(25n^2,\z y_1+\z(5y_2)+\z (-2y_2+y_3)+ \z (-y_2+y_4))\\
             &+r(25n^2,\z y_1+\z(5y_2)+\z (y_2+y_3)+ \z (2y_2+ y_4))\\
             &-3r(25n^2,\z(5y_1)+\z y_2+\z (5y_3)+ \z (-y_3+y_4),
\end{array}
$$
which implies that 
$$
r(25n^2,L_3')=2r(5n^2,M)+2r(25n^2,K_2)-3r(n^2,L_3'),
$$
where $K_2=\langle1\rangle\perp\begin{pmatrix}6&6&10\\6&21&35\\10&35&75\end{pmatrix}$. Similarly, we have
$$
r(25n^2,K_2)=2r(5n^2,N)-r(n^2,L_3').
$$
Hence, we have
\begin{equation}\label{e4}
r(25n^2,L_3')=2r(5n^2,M)+4r(5n^2,N)-5r(n^2,L_3').
\end{equation}
Therefore, by \eqref{145}, \eqref{e3} and \eqref{e4}, the genus $\gen(L_3)$ is indistinguishable by squares and $L_3$ and $L_3'$ are strongly $s$-regular.

\end{proof}

\end{document}